\documentclass[12pt,twoside]{amsart}

\usepackage{times,a4wide,amssymb,amsmath,amsthm,}
\usepackage[draft=false]{hyperref}
\usepackage[alphabetic]{amsrefs}

\theoremstyle{plain}
\newtheorem{thm}{Theorem}
\newtheorem{prop}[thm]{Proposition}

\newtheorem{lemma}[thm]{Lemma}

\theoremstyle{remark}
\newtheorem*{rem}{Remark}
\newtheorem*{rems}{Remarks}
\theoremstyle{definition}

\newcommand{\C}{{\mathbb C}}

\newcommand{\D}{{\mathbb D}}

\newcommand{\E}{{\mathbb E}}
\newcommand{\V}{{\mathbb V}}

\newcommand{\Abs}[1]{\left|#1\right|}

\newcommand{\dl}{\operatorname{Li_2}}

\DeclareMathOperator{\cov}{Cov}

\title{Fluctuations in the zero set of the hyperbolic Gaussian analytic function}
\thanks{Supported by the Generalitat de Catalunya (grant 2009 SGR 1303) and the Spanish Ministerio de Econom\'ia y Competividad (project MTM2011-27932-C02-01)}
\author{Jeremiah Buckley}
\address{Dept.\ Matem\`atica Aplicada i An\`alisi, Universitat  de Barcelona, Gran Via 585, 08007 Bar\-ce\-lo\-na, Spain}
\email{jerry.buckley@ub.edu}
\begin{document}
\begin{abstract}
The zero set of the hyperbolic Gaussian analytic function is a random point process in the unit disc whose distribution is invariant under automorphisms
of the disc. We study the variance of the number of points in a disc of increasing radius. Somewhat surprisingly, we find a change of behaviour at a
certain value of the `intensity' of the process, which appears to be novel.
\end{abstract}
\maketitle

\section{Introduction and statement of results}
The hyperbolic Gaussian analytic function (GAF) is a random holomorphic function on the unit disc. This GAF is particularly interesting because the
distribution of its zero set is invariant under disc automorphisms. We begin with the definition and some elementary properties, further details and
proofs of these facts may be found in \cite{HKPV}. Fix a parameter $L>0$ and define
\[
f_L(z)=\sum_{n=0}^\infty a_n\binom{L+n-1}{n}^{1/2}z^n
\]
for $z\in\D$, where $(a_n)_{n=0}^\infty$ is a sequence of iid standard complex normal random variables, and
\[
\binom{L+n-1}{n}=\frac{\Gamma(L+n)}{\Gamma(n+1)\Gamma(L)}=\frac{L(L+1)\cdots(L+n-1)}{n!}.
\]
This sum almost surely defines a holomorphic function in the unit disc with associated covariance kernel
\[
K_L(z,w)=\E[f_L(z)\overline{f_L(w)}]=(1-z\overline{w})^{-L}.
\]
Moreover, the distribution of the zero set of $f_L$ is invariant under automorphisms of the disc, and $(f_L)_{L>0}$ are (essentially) the only GAFs with
this property. We denote the counting measure on the zero set of $f_L$ by $n_L$ and write $n_L(r)=n_L(D(0,r))$ to simplify the notation. The
Edelman-Kostlan formula yields
\begin{equation*}
\E[dn_L]=\frac1{4\pi}\Delta\log K_L(z,z)dm(z)=\frac{L}{\pi}\frac{dm(z)}{(1-|z|^2)^2}
\end{equation*}
so that the mean number of zeroes is given by $L$ times the hyperbolic measure (normalised appropriately). This means that we can think of the parameter
$L$ as corresponding to the `intensity' of the process. In particular
\begin{equation*}
\E[n_L(r)]=\frac{Lr^2}{1-r^2}.
\end{equation*}

We will be interested in the size of fluctuations of the zero set. There exists an analogous process in the plane, the zero set of the flat GAF whose
distribution is invariant under plane automorphisms. The flat GAF is also defined in terms of a parameter $L>0$ which can be thought of as corresponding
to the intensity of the process, and the mean number of zeroes is given by $L/\pi$ times the Lebesgue measure. Forrester and Honner \cite{FH} found that
the variance of the number of points in a set $D$ with piecewise smooth boundary is given by
\begin{equation}\label{ForHonres}
\frac{\zeta(3/2)}{8\pi^{3/2}} \sqrt{L}|\partial D|(1+o(1))
\end{equation}
as $\sqrt{L}|\partial D|\to\infty$, where $\zeta$ is the usual Riemann $\zeta$-function and $|\cdot|$ denotes the length; later Nazarov and Sodin
\cite{NS}*{Theorem 1.1} computed the variance exactly. Shiffman and Zelditch \cite{SZ08}*{Theorem 1.1} derived analogous formulae in compact m-dimensional
manifolds.

In the planar setting a dilation of the plane allows one to consider only the case $L=1$, but this does not hold in the hyperbolic case. It seems that in
the hyperbolic case, for large values of $L$ but for a fixed set $D$, a result identical to \eqref{ForHonres}, replacing $|\cdot|$ by the hyperbolic
length, is folkloric. We shall instead consider fixed $L$ and study the variance of the number of points in a disc of radius $r\rightarrow1^-$.

For one particular value of the intensity, $L=1$, Peres and Vir\'{a}g \cite{PV}*{Theorem 2} have completely described the distribution of the random
variable $n_1(r)$, and we recover \cite{PV}*{Corollary 3 (iii)}. Their results were proved by showing that the corresponding zero set is a determinantal
process, but this holds for no other value of $L$. Since we are interested in the full range of $L$, our techniques are accordingly quite different.

In this paper we compute the variance of $n_L(r)$ as $r\rightarrow1^-$, in various regimes of $L$. One feature to emerge from our computations is a change
of behaviour at $L=1/2$. This appears to be novel; we do not know of any other properties of the zero set that change at $L=1/2$. This may deserve further
investigation.

We write $o(1)$ to denote a quantity that can be made arbitrarily small as $r$ approaches $1$ but that may depend on $L$ unless explicitly stated
otherwise.

Our first result is the following.
\begin{thm}\label{varrto1}
(a) For each fixed $L>1/2$, as $r\rightarrow1^-$,
\[
\V[n_L(r)]= \frac{c_L}{1-r} (1+o(1)),
\]
where
\[
c_L=\frac{L^2}{2\pi}\int_{0}^{\infty} \frac1{(1+x^2)^L-1} \frac{x^2}{1+x^2}\,dx=\frac{L^2}{8\sqrt{\pi}}\sum_{n=1}^{\infty}
\frac{\Gamma(Ln-\frac12)}{\Gamma(Ln+1)}.
\]
Moreover the quantity $o(1)$ can be taken to be uniform in $L$ for all $L>1$.

(b) We have, as $r\rightarrow1^-$,
\[
\V[n_{1/2}(r)]=\frac1{8\pi} \frac{1}{1-r}\log\frac{1}{1-r} (1+o(1)).
\]

(c) For each fixed $L<1/2$, as $r\rightarrow1^-$,
\[
\V[n_L(r)]= \frac{c_L}{(1-r)^{2-2L}} (1+o(1)),
\]
where
\[
c_L=\frac{L^2\Gamma(\frac12-L)}{4\sqrt{\pi} \Gamma(1-L)}.
\]
\end{thm}
We also compute the behaviour of the variance for large $L$, for $L$ close to the critical value $1/2$ and for $L\to0$.
\begin{thm}\label{critvals}
(a) We have, as $L\to\infty$ and $r\to1^-$,
\[
\V[n_L(r)]=\frac{\zeta(3/2)}{8\sqrt{\pi}}\frac{\sqrt{L}}{1-r}(1+o(1))
\]
where the term $o(1)$ is uniform in $L$ and $r$. In other words
\[
\lim_{\substack{L\to\infty\\r\to1^-}}\frac{1-r}{\sqrt{L}}\V[n_L(r)]=\frac{\zeta(3/2)}{8\sqrt{\pi}}
\]
independent of the manner in which $L\to\infty$ and $r\to1^-$.

(b) We have
\[
\V[n_L(r)]= \frac{1-(1-r)^{2L-1}}{8\pi(2L-1)(1-r)}(1+o(1))
\]
as $L\rightarrow1/2^+$ and $r\rightarrow1^-$, where the quantity $o(1)$ is uniform in $L$ and $r$.

(c) If $L\rightarrow1/2^-$ and $r\rightarrow1^-$ then
\[
\V[n_L(r)]= \frac{1-(1-r)^{1-2L}}{8\pi(1-2L)(1-r)^{2-2L}} (1 + o(1))
\]
where the quantity $o(1)$ is uniform in both $L$ and $r$.

(d) If $L\rightarrow0^+$, $r\rightarrow1^-$ and $\frac L{1-r}\to\infty$ then
\[
I_L(r)= \frac{L^2}4\frac{(1-r)^{2L-2}}{1-(1-r)^{2L}} (1 + o(1))
\]
where the quantity $o(1)$ is uniform in $L$, $r$, and $\frac L{1-r}$.
\end{thm}
\begin{rems}
1. Noting that the hyperbolic length of the circle of radius $r$ is given by $\frac{2\pi r}{1-r^2}$ we see that (a) is consistent with replacing $|\cdot|$
by the hyperbolic length in \eqref{ForHonres}.

2. We impose the condition $\frac L{1-r}\to\infty$ in (d) because it is equivalent to $\E[n_L(r)]\to\infty$.
\end{rems}
In the particular cases $L=1,2$ we can show the following more precise result.
\begin{thm}\label{varLis1or2}
For any $0<r<1$
\[
\V[n_1(r)]=\frac{r^2}{1-r^4}
\]
and
\[
\V[n_2(r)]=\frac{4r^2}{1-r^2}\left(\frac1{1+r^2}-\frac1{2\sqrt{1+r^4}}\right)
\]
\end{thm}
\begin{rem}
The $L=1$ result was first given by Peres and Vir\'{a}g \cite{PV}*{Corollary 3 (iii)}, as we mentioned earlier.
\end{rem}

The paper is structured as follows: In Section~\ref{key} we reduce the computation of the variance of $n_L(r)$ to the evaluation of an integral of a
positive function of one real variable; this reduction is the main ingredient in our work. In Section~\ref{L12} we prove Theorem~\ref{varLis1or2} by
computing this integral exactly for $L=1$ and $L=2$. In Section~\ref{asymptotics} we prove Theorems~\ref{varrto1} and \ref{critvals} by computing the
asymptotics of this integral as $r\to1^-$.

We shall use the following standard notation: The expression $f \lesssim g$ means that there is a constant $C$ independent of the relevant variables such
that $f\le C g$, and $f\simeq g$ means that $f\lesssim g$ and $g\lesssim f$. We sometimes write $f=O(g)$ to mean $|f| \lesssim g$.

\section{The key lemma}\label{key}
In this section we prove a lemma which allows us to compute the variance of $n_L(r)$ by evaluating an integral of a positive function of one real
variable. The starting point in our computations is the following formula (see \cite{SZ08}*{Theorem~3.1} or \cite{NS}*{Lemma~2.3}). For any (sufficiently
nice) $D\subset\D$ we have
\begin{equation}\label{varform}
\V[n_L(D)]=\int_{D}\int_{D} \Delta_z\Delta_w \frac14\dl(J_L(z,w)) \frac{dm(z)}{2\pi} \frac{dm(w)}{2\pi}
\end{equation}
where we define the dilogarithm
\[
\dl(\zeta)=\sum_{n=1}^\infty\frac{\zeta^n}{n^2}
\]
and
\[
J_L(z,w)=\frac{|K_L(z,w)|^2}{K_L(z,z)K_L(w,w)}=J_1(z,w)^L
\]
where
\[
J_1(z,w)=\frac{(1-|z|^2)(1-|w|^2)}{|1-z\overline{w}|^2}=1-\Abs{\frac{z-w}{1-z\overline{w}}}^2.
\]

For completeness, we will sketch a proof of this: Detailed computations can be found in \cite{SZ08}*{Section 3} or \cite{NS}*{Section 2.1}. Green's
formula implies that
\[
dn_L= \frac{1}{2\pi} \Delta \log|f_L|\,dm
\]
(this equality is to be understood in the distributional sense) which combined with the Edelman-Kostlan formula gives
\[
d\tilde{n}_L=dn_L-\E\left[dn_L\right]= \Delta\log|\hat{f}_L(z)|\frac{dm(z)}{2\pi}.
\]
where $\hat{f}_L(z)=\frac{f_L(z)}{K_L(z,z)^{1/2}}$. Thus
\[
\E\left[d\tilde{n}_L\times d\tilde{n}_L\right] = \Delta_z\Delta_w \E\left[\log|\hat{f}_L(z)| \log|\hat{f}_L(w)|\right] \frac{dm(z)}{2\pi}
\frac{dm(w)}{2\pi},
\]
where the exchange of expectation and the Laplacians is justified in the distributional sense by integrating against smooth compactly supported
(deterministic) test functions. Note that $\hat{f}_L(z)$ is a $\mathcal{N}_\C(0,1)$ random variable for each $z\in\C$. Thus $\E[\log|\hat{f}_L(z)|]$ is
independent of $z$ and so
\[
\Delta_z \E[\log|\hat{f}_L(z)|]= \Delta_w \E[\log|\hat{f}_L(w)|]=0.
\]
Thus
\[
\E\left[d\tilde{n}_L\times d\tilde{n}_L\right] = \Delta_z \Delta_w \cov\left[\log|\hat{f}_L(z)|, \log|\hat{f}_L(w)|\right] \frac{dm(z)}{2\pi}
\frac{dm(w)}{2\pi}
\]
where $\cov$ indicates covariance. We may therefore apply the following lemma.
\begin{lemma}[\citelist{\cite{SZ08}*{Lemma 3.3} \cite{NS}*{Lemma 2.2}}]
If $\zeta_1$ and $\zeta_2$ are $\mathcal{N}_\C(0,1)$ random variables with $\E[\zeta_1\bar{\zeta_2}]=\theta$ then
\[
\cov[\log|\zeta_1|,\log|\zeta_2|]= \frac14 \dl(|\theta|^2)
\]
\end{lemma}
Noting that
\[
\E\left[\hat{f}_L(z)\overline{\hat{f}_L(w)}\right]= \frac{K_L(z,w)}{K_L(z,z)^{1/2}K_L(w,w)^{1/2}}
\]
and applying the lemma we get
\[
\E\left[d\tilde{n}_L\times d\tilde{n}_L\right] = \frac14 \Delta_z\Delta_w \dl(J_L(z,w))\frac{dm(z)}{2\pi}\frac{dm(w)}{2\pi},
\]
as a distribution. It remains only to see that $\V[n_L(D)]=\E[\tilde{n}_L(D)^2]$.

We now use this formula to prove our key lemma.
\begin{lemma}\label{rdecomp}
For any $0<r<1$
\[
\V[n_L(r)]=\frac {L^2r^4}{2\pi(1-r^2)^2}I_L(r),
\]
where
\[
I_L(r)=\int_{-\pi}^\pi \frac{(1-r^2)^{2L}} {|1-r^2e^{i\theta}|^{2L}-(1-r^2)^{2L}} \frac {2(1-\cos\theta)} {|1-r^2e^{i\theta}|^{2}}\,d\theta.
\]
\end{lemma}
\begin{proof}
For any $D\subset\D$ with piecewise smooth boundary, applying Stokes' Theorem to \eqref{varform} we get
\begin{align*}
\V[n_L(D)]&=\int_{D}\int_{D} \Delta_z\Delta_w \frac14\dl(J_L(z,w)) \frac{dm(z)}{2\pi} \frac{dm(w)}{2\pi}\\
&=-\frac {1}{4\pi^2} \int_{\partial D}\int_{\partial D} \frac{\partial}{\partial \bar{z}} \frac{\partial}{\partial \bar{w}} \dl(J_L(z,w))
\,d\bar{z}d\bar{w}.
\end{align*}
Recalling that the dilogarithm satisfies
\[
\frac{d}{d\zeta}\dl(\zeta)=\frac1\zeta\log\frac1{1-\zeta}.
\]
we have
\[
\frac{\partial}{\partial \bar{w}}\dl(J_L(z,w)) =\frac1{J_L} \log\frac1{1-J_L} \frac{\partial J_L}{\partial \bar{w}}= \frac LJ_1 \log\frac1{1-J_L}
\frac{\partial J_1}{\partial \bar{w}}
\]
and so
\begin{align*}
\frac{\partial}{\partial \bar{z}}\frac{\partial}{\partial \bar{w}}\dl(J_L(z,w))&=\frac {L^2}{J_1^2} \frac{J_L}{1-J_L} \frac{\partial J_1}{\partial
\bar{z}} \frac{\partial J_1}{\partial \bar{w}} + L \log\frac1{1-J_L} \left(\frac 1{J_1} \frac{\partial^2J_1}{\partial
\bar{z} \partial \bar{w}}-\frac 1{J_1^2} \frac{\partial J_1}{\partial \bar{z}} \frac{\partial J_1}{\partial \bar{w}}\right)
\end{align*}
Routine but tedious calculations yield
\[
\frac{\partial J_1}{\partial \bar{z}} = \frac{1-|w|^2}{|1-z\overline{w}|^2} \frac{w-z}{1-\overline{z}w}
\]
\[
\frac{\partial J_1}{\partial \bar{w}} = \frac{1-|z|^2}{|1-z\overline{w}|^2} \frac{z-w}{1-z\overline{w}}
\]
and
\[
\frac{\partial^2J_1}{\partial \bar{z} \partial \bar{w}}=-\frac{(z-w)^2}{|1-z\overline{w}|^4},
\]
so that
\[
\frac 1{J_1} \frac{\partial^2J_1}{\partial \bar{z} \partial \bar{w}}-\frac 1{J_1^2} \frac{\partial J_1}{\partial \bar{z}} \frac{\partial J_1}{\partial
\bar{w}}=0.
\]
We conclude that
\[
\V[n_L(D)]=\frac {L^2}{4\pi^2} \int_{\partial D}\int_{\partial D} \frac{J_L}{1-J_L} \frac1{(1-|z|^2)(1-|w|^2)} \frac{(z-w)^2}{|1-z\overline{w}|^2}
\,d\bar{z}d\bar{w}.
\]

We now suppose that $D=D(0,r)$ for $r<1$. Then, writing $z=r e^{i\theta}$ and $w=r e^{i\phi}$, after some simplifications we have
\[
\V[n_L(D)]=\frac {L^2r^4}{4\pi^2(1-r^2)^2} \int_{-\pi}^\pi \int_{-\pi}^\pi \frac{(1-r^2)^{2L}} {|1-r^2e^{i(\theta-\phi)}|^{2L}-(1-r^2)^{2L}} \frac
{2(1-\cos(\theta-\phi))} {|1-r^2e^{i(\theta-\phi)}|^{2}}\,d\theta d\phi
\]
We note that the integrand depends on the difference $\theta-\phi$, so one of the integrals immediately evaluates to $2\pi$. We are left with
\[
\V[n_L(D)]=\frac {L^2r^4}{2\pi(1-r^2)^2}\int_{-\pi}^\pi \frac{(1-r^2)^{2L}} {|1-r^2e^{i\theta}|^{2L}-(1-r^2)^{2L}} \frac {2(1-\cos\theta)}
{|1-r^2e^{i\theta}|^{2}}\,d\theta.
\]
as claimed.
\end{proof}

\section{Proof of Theorem \ref{varLis1or2}}\label{L12}
In this section we prove Theorem \ref{varLis1or2}. By Lemma~\ref{rdecomp} we need only compute $I_1(r)$ and $I_2(r)$, which we do in the following
proposition.
\begin{prop}\label{rdecompLis1or2}
For any $0<r<1$
\[
I_1(r)=\frac{2\pi(1-r^2)}{r^2(1+r^2)}
\]
and
\[
I_2(r)=\frac{2\pi(1-r^2)}{r^2}\left(\frac1{1+r^2}-\frac1{2\sqrt{1+r^4}}\right)
\]
\end{prop}
\begin{proof}
We first suppose that $L$ is an integer. Then
\begin{align*}
I_L(r)&=\int_{-\pi}^\pi \frac{(1-r^2)^{2L}} {|1-r^2e^{i\theta}|^{2L}-(1-r^2)^{2L}} \frac {2(1-\cos\theta)} {|1-r^2e^{i\theta}|^{2}}\,d\theta\\
&=\int_{\partial\D} \frac{(1-r^2)^{2L}} {(1-r^2z)^{L}(1-r^2/z)^{L}-(1-r^2)^{2L}} \frac {(1-z)(1-1/z)} {(1-r^2z)(1-r^2/z)} \frac{dz}{iz}\\
&=\int_{\partial\D} -\frac1i\frac{(1-r^2)^{2L}z^{L-1}} {(1-r^2z)^{L}(z-r^2)^{L}-z^L(1-r^2)^{2L}} \frac {(1-z)^2} {(1-r^2z)(z-r^2)}\,dz.
\end{align*}
We note that the integrand has simple poles at $r^{-2}$ which lies outside the disc, and at $r^2$ with residue
\[
\frac1{ir^2} \frac{1-r^2}{1+r^2}.
\]
Finally there are poles at the zeroes of the polynomial
\[
(1-r^2z)^{L}(z-r^2)^{L}-z^L(1-r^2)^{2L}.
\]
This is equivalent to finding the zeroes of
\begin{equation}\label{quadomega}
(1-r^2z)(z-r^2)-\omega z(1-r^2)^{2}.
\end{equation}
for each $L$th root of unity $\omega$ satisfying $\omega^L=1$, which is in turn equivalent to finding the zeroes of
\begin{equation}\label{quadomega2}
(z-1)^2-\frac{1-\omega}{r^2}(1-r^2)z.
\end{equation}

Now if $L=1$ we have only $\omega=1$, and there is a double zero at $1$. This pole is removable, since there is a factor $(1-z)^2$ in the numerator of the
integrand. We conclude that
\[
I_1(r)=\frac{2\pi(1-r^2)}{r^2(1+r^2)}.
\]

If $L=2$ we have $\omega=1,-1$. Once more if $\omega=1$ there is a double zero at $1$ which gives a removable pole. If $\omega=-1$ we can solve
\eqref{quadomega2} easily and get two distinct zeroes
\[
z^{(i)}=1+\frac{(1-r^2)^2}{r^2}-\frac{1-r^2}{r^2}\sqrt{1+r^4}
\]
which is inside the unit disc and
\[
z^{(o)}=1+\frac{(1-r^2)^2}{r^2}+\frac{1-r^2}{r^2}\sqrt{1+r^4}
\]
outside. This yields
\[
(1-r^2z)^{2}(z-r^2)^{2}-z^2(1-r^2)^{4}=r^4(z-1)^2(z-z^{(i)})(z-z^{(o)})
\]
and so the integrand simplifies to
\[
-\frac1i\frac{(1-r^2)^{4}z} {r^4(1-r^2z)(z-r^2)} \frac {1} {(z-z^{(i)})(z-z^{(o)})}.
\]
Noting that $z_\omega^{(i)}$ is a zero of \eqref{quadomega}, we compute the residue at $z_\omega^{(i)}$ to be
\[
-\frac1i\frac{(1-r^2)^{4}} {-r^4(1-r^2)^2} \frac {1} {(z^{(i)}-z^{(o)})}=-\frac{1-r^2} {2ir^2\sqrt{1+r^4}}
\]
which gives
\[
I_2(r)=2\pi\left(\frac1{r^2} \frac{1-r^2}{1+r^2}-\frac{1-r^2} {2r^2\sqrt{1+r^4}}\right)
=\frac{2\pi(1-r^2)}{r^2}\left(\frac1{1+r^2}-\frac1{2\sqrt{1+r^4}}\right).
\]
\end{proof}
\begin{rems}
1. If we are only interested in the case $L=1$, we may compute $I_1(r)$ without recourse to residue calculus. First note that the integrand simplifies to
$|1-r^2e^{i\theta}|^{-2}$ (and some factors that depend on $r$). From the geometric series we have
\[
|1-r^2e^{i\theta}|^{-2}=\sum_{n,m=0}^\infty r^{2(n+m)}e^{i\theta (n-m)}.
\]
Integrating this expression term by term yields the result.

2. In principle, this should be computable for any integer $L$. We need to compute the zeroes of \eqref{quadomega2}. Since the product of the zeroes is
$1$ and the sum of the zeroes is $2+\frac{1-\omega}{r^2}(1-r^2)$ which has real part strictly greater than $2$, we see that there are two distinct zeroes,
one inside the disc and one outside, which we label $z_\omega^{(i)}$ and $z_\omega^{(o)}$ respectively. We therefore have
\[
(1-r^2z)^{L}(z-r^2)^{L}-z^L(1-r^2)^{2L}=(-r^2)^{L}(z-1)^2\prod_\omega(z-z_\omega^{(i)})(z-z_\omega^{(o)})
\]
where the product ranges over the $L-1$ non-trivial roots of unity. Thus the integrand simplifies to
\[
-\frac1i\frac{(1-r^2)^{2L}z^{L-1}} {(-r^2)^{L}(1-r^2z)(z-r^2)} \frac {1} {\prod_\omega(z-z_\omega^{(i)})(z-z_\omega^{(o)})}.
\]
Noting that $z_\omega^{(i)}$ is a zero of \eqref{quadomega}, we compute the residue at $z_\omega^{(i)}$ to be
\begin{align*}
-\frac1i \frac{(1-r^2)^{2L}(z_\omega^{(i)})^{L-1}} {(-r^2)^{L}(1-r^2z_\omega^{(i)})(z_\omega^{(i)}-r^2)} \frac {1} {(z_\omega^{(i)}-z_\omega^{(o)})
\prod_{\widetilde{\omega}\neq\omega}(z_\omega^{(i)}-z_{\widetilde{\omega}}^{(i)})(z_\omega^{(i)}-z_{\widetilde{\omega}}^{(o)})}\\
= -\frac1i \frac{(1-r^2)^{2L-2}(z_\omega^{(i)})^{L-2}} {(-r^2)^{L}\omega} \frac {1} {(z_\omega^{(i)}-z_\omega^{(o)})
\prod_{\widetilde{\omega}\neq\omega}(z_\omega^{(i)}-z_{\widetilde{\omega}}^{(i)})(z_\omega^{(i)}-z_{\widetilde{\omega}}^{(o)})}
\end{align*}
and so we conclude that
\[
I_L(r)=2\pi\left(\frac1{r^2} \frac{1-r^2}{1+r^2}-\sum_\omega\frac{(1-r^2)^{2L-2}(z_\omega^{(i)})^{L-2}} {(-r^2)^{L}\omega} \frac {1}
{(z_\omega^{(i)}-z_\omega^{(o)})
\prod_{\widetilde{\omega}\neq\omega}(z_\omega^{(i)}-z_{\widetilde{\omega}}^{(i)})(z_\omega^{(i)}-z_{\widetilde{\omega}}^{(o)})}\right).
\]

From here the algebra seems intractable and we have contented ourselves with considering only the values $L=1,2$. Mathematica yields an explicit
expression for $L=4$, however we have not been persistent enough to establish its veracity.

3. Mathematica also yields a closed expression if $L=1/2$ in terms of some special function, that is not terribly enlightening.
\end{rems}

\section{Proof of Theorems \ref{varrto1} and \ref{critvals}}\label{asymptotics}
In this section we prove Theorems \ref{varrto1} and \ref{critvals}. By Lemma~\ref{rdecomp} we need only compute the asymptotic behaviour of $I_L(r)$ as
$r\rightarrow1^-$. By examining the integrand it is clear that for $\theta$ smaller than $1-r^2$ the integrand is approximately constant, so we get a
contribution of size $(1-r^2)$. However if $|\theta|$ is close to $\pi$ the integrand is approximately $(1-r^2)^{2L}$. The important region of integration
therefore depends on whether or not $L>1/2$. The next proposition makes this reasoning precise, and completes the proof of Theorems~\ref{varrto1} and
\ref{critvals}~(b)--(d).
\begin{prop}
(a) For each fixed $L>1/2$, as $r\rightarrow1^-$,
\begin{align*}
I_L(r)&= 4(1-r) \int_{0}^{\infty} \frac1{(1+x^2)^L-1} \frac{x^2}{1+x^2}\,dx (1+o(1))\\
&=\sqrt{\pi}(1-r)\sum_{n=1}^{\infty} \frac{\Gamma(Ln-\frac12)}{\Gamma(Ln+1)}(1+o(1)).
\end{align*}
Moreover the quantity $o(1)$ can be taken to be uniform in $L$ for all $L>1$.

Furthermore we have
\[
I_L(r)= \frac{2(1-r)}{L-\frac12}\left(1-(1-r)^{2L-1}\right)(1+o(1))
\]
as $L\rightarrow1/2^+$ and $r\rightarrow1^-$, where the quantity $o(1)$ is uniform in $L$ and $r$.

(b) We have, as $r\rightarrow1^-$,
\[
I_{1/2}(r)=4(1-r)\log\frac{1}{1-r} (1+o(1)).
\]

(c) For each fixed $L<1/2$,
\[
I_L(r)=\frac{2\sqrt{\pi}\Gamma(\frac12-L)}{\Gamma(1-L)} (1-r)^{2L} (1+o(1)),
\]
as $r\rightarrow1^-$.

If $L\rightarrow1/2^-$ and $r\rightarrow1^-$ then
\[
I_L(r)= \frac{2(1-r)^{2L}}{\frac12-L}(1-(1-r)^{1-2L}) (1 + o(1))
\]
where the quantity $o(1)$ is uniform in both parameters.

We have
\[
I_L(r)= 2\pi\frac{(1-r)^{2L}}{1-(1-r)^{2L}} (1 + o(1))
\]
as $L\rightarrow0^+$, $r\rightarrow1^-$ and $\frac L{1-r}\to\infty$, where the quantity $o(1)$ is uniform in $L$, $r$, and $\frac L{1-r}$.
\end{prop}
\begin{proof}
We first note that
\[
|1-r^2e^{i\theta}|^{2}=(1-r^2)^2+2r^2(1-\cos\theta)
\]
and so, from Lemma~\ref{rdecomp}
\begin{align*}
I_L(r)&=\int_{-\pi}^\pi\frac{(1-r^2)^{2L}} {|1-r^2e^{i\theta}|^{2L}-(1-r^2)^{2L}} \frac {2(1-\cos\theta)} {|1-r^2e^{i\theta}|^{2}}\,d\theta\\
&=\frac 2 {r^2}\int_{0}^\pi\left(\left(1+2r^2\frac{1-\cos\theta}{(1-r^2)^2}\right)^{L}-1\right)^{-1}
\left(1+\frac{(1-r^2)^2}{2r^2(1-\cos\theta)}\right)^{-1} \,d\theta.
\end{align*}
Making the change of variables $x=\frac{2r^2}{(1-r^2)^2}(1-\cos\theta)$ we see that
\begin{equation}\label{intinx}
I_L(r)= \frac {1-r^2} {r^3}\int_{0}^{\frac{4r^2}{(1-r^2)^2}} \frac1{(1+x)^L-1} \frac{\sqrt{x}}{1+x} \frac{dx}{\sqrt{1-\frac{(1-r^2)^2}{4r^2}x}}
\end{equation}

(a) We first assume that $L>1/2$. Bearing in mind the remarks preceding the statement of this lemma, we expect the main contribution to come from the
`small' values of $x$. Now
\begin{align*}
\int_{0}^{\left(\log\frac1{1-r}\right)^{-1}\frac{4r^2}{(1-r^2)^2}} \frac1{(1+x)^L-1} \frac{\sqrt{x}}{1+x} \frac{dx}{\sqrt{1-\frac{(1-r^2)^2}{4r^2}x}} &\\
=\int_{0}^{\left(\log\frac1{1-r}\right)^{-1}\frac{4r^2}{(1-r^2)^2}} \frac1{(1+x)^L-1} \frac{\sqrt{x}}{1+x}&\,dx (1+o(1))
\end{align*}
where the term $o(1)$ is uniform in $L$. Trivially
\begin{align*}
\int_{0}^{\left(\log\frac1{1-r}\right)^{-1}\frac{4r^2}{(1-r^2)^2}} \frac1{(1+x)^L-1} \frac{\sqrt{x}}{1+x}\,dx &\\ =\int_{0}^{\infty} \frac1{(1+x)^L-1}
\frac{\sqrt{x}}{1+x}&\,dx - \int_{\left(\log\frac1{1-r}\right)^{-1}\frac{4r^2}{(1-r^2)^2}}^{\infty} \frac1{(1+x)^L-1} \frac{\sqrt{x}}{1+x}\,dx,
\end{align*}
and these integrals clearly converge in this range of $L$. The change of variables $t=\sqrt{x}$ yields
\[
\int_{0}^{\infty} \frac1{(1+x)^L-1} \frac{\sqrt{x}}{1+x}\,dx = 2\int_{0}^{\infty} \frac1{(1+t^2)^L-1} \frac{t^2}{1+t^2}\,dt.
\]
The alternative change of variables $s=(1+x)^{-1}$ gives us ($B$ denotes the usual Beta function)
\begin{align*}
\int_{0}^{\infty} \frac1{(1+x)^L-1} \frac{\sqrt{x}}{1+x}\,dx &= \int_{0}^{1} \frac1{s^{-L}-1} \sqrt{\frac1s-1} \frac{ds}{s}\\
&= \int_{0}^{1} s^{L-\frac32}\frac1{1-s^{L}} \sqrt{1-s}\,ds\\
&= \sum_{n=0}^\infty \int_{0}^{1} s^{L(n+1)-\frac32} \sqrt{1-s}\,ds\\
&= \sum_{n=0}^\infty B\left(L(n+1)-\frac12,\frac32\right)\\
&= \sum_{n=0}^\infty \frac{\Gamma(L(n+1)-\frac12)\Gamma(\frac32)}{\Gamma(L(n+1)+1)}\\
&= \frac{\sqrt{\pi}}{2}\sum_{n=1}^\infty \frac{\Gamma(Ln-\frac12)}{\Gamma(Ln+1)}.
\end{align*}
As $L\rightarrow1/2^+$ we have
\begin{equation*}
\frac{\sqrt{\pi}}{2}\sum_{n=1}^\infty \frac{\Gamma(Ln-\frac12)}{\Gamma(Ln+1)} =
\frac{\sqrt{\pi}}{2}\left(\frac{\Gamma(L-\frac12)}{\Gamma(L+1)}+\sum_{n=2}^\infty \frac{\Gamma(Ln-\frac12)}{\Gamma(Ln+1)}\right) = \frac1{L-\frac12}+O(1)
\end{equation*}
since $\Gamma(z)$ has a simple pole with residue $1$ at $z=0$. Now, for a fixed value of $L$,
\begin{equation*}
\int_{\frac{1}{\log\frac1{1-r}}\frac{4r^2}{(1-r^2)^2}}^{\infty} \frac1{(1+x)^L-1} \frac{\sqrt{x}}{1+x}\,dx = o(1).
\end{equation*}
Moreover since
\[
\frac1{(1+x)^L-1}\leq\frac1{(1+x)^M-1}
\]
for $L\geq M$ and $x>0$ the term $o(1)$ may be taken to be uniform in $L$ for all $L\geq 1$ (say). As $L\rightarrow1/2^+$ we have
\begin{align*}
\int_{\left(\log\frac1{1-r}\right)^{-1}\frac{4r^2}{(1-r^2)^2}}^{\infty} \frac1{(1+x)^L-1} \frac{\sqrt{x}}{1+x}\,dx &=
\int_{\left(\log\frac1{1-r}\right)^{-1}\frac{4r^2}{(1-r^2)^2}}^{\infty} \frac1{x^{L+1/2}} + O\left(\frac1{x^{2L+1/2}}\right) \,dx\\
&=\frac{(1-r)^{2L-1}}{L-\frac12}\left(\log\frac1{1-r}\right)^{L-\frac12}(1+o(1)).
\end{align*}
We therefore have
\begin{equation*}
\int_{0}^{\left(\log\frac1{1-r}\right)^{-1}\frac{4r^2}{(1-r^2)^2}} \frac1{(1+x)^L-1} \frac{\sqrt{x}}{1+x} \frac{dx}{\sqrt{1-\frac{(1-r^2)^2}{4r^2}x}}
=\frac{\sqrt{\pi}}{2}\sum_{n=1}^\infty \frac{\Gamma(Ln-\frac12)}{\Gamma(Ln+1)} (1+o(1))
\end{equation*}
for a fixed value of $L$, where the term $o(1)$ is uniform in $L$ for all $L\geq 1$, while
\begin{align*}
\int_{0}^{\left(\log\frac1{1-r}\right)^{-1}\frac{4r^2}{(1-r^2)^2}} &\frac1{(1+x)^L-1} \frac{\sqrt{x}}{1+x} \frac{dx}{\sqrt{1-\frac{(1-r^2)^2}{4r^2}x}}\\
&=\frac1{L-\frac12}\left(1-(1-r)^{2L-1}\left(\log\frac1{1-r}\right)^{L-\frac12}\right)(1+o(1))+O(1)\\
&=\frac1{L-\frac12}\left(1-(1-r)^{2L-1}\right)(1+o(1))
\end{align*}
as $L\rightarrow1/2^+$.

We now show that the remaining contributions to \eqref{intinx} are negligible in comparison. We have, making the change of variables
$y=\frac{(1-r^2)^2}{4r^2}x$,
\begin{align*}
\int_{\left(\log\frac1{1-r}\right)^{-1}\frac{4r^2}{(1-r^2)^2}}^{\frac{4r^2}{(1-r^2)^2}} \frac1{(1+x)^L-1} \frac{\sqrt{x}}{1+x}
\frac{dx}{\sqrt{1-\frac{(1-r^2)^2}{4r^2}x}} &\simeq \int_{\left(\log\frac1{1-r}\right)^{-1}\frac{4r^2}{(1-r^2)^2}}^{\frac{4r^2}{(1-r^2)^2}} \frac1{x^{L+1/2}} \frac{dx}{\sqrt{1-\frac{(1-r^2)^2}{4r^2}x}}\\
&= (1-r^2)^{2L-1}\int_{\left(\log\frac1{1-r}\right)^{-1}}^1 \frac 1{y^{L+1/2}} \frac{dy}{\sqrt{1-y}}\\
&\simeq \frac{(1-r^2)^{2L-1}}{L-1/2}\left(\left(\log\frac1{1-r}\right)^{L-1/2}-1\right)
\end{align*}
which is easily seen to be $o(1)$ for fixed $L$. Moreover, using once more the fact that
\[
\frac1{(1+x)^L-1}\leq\frac1{(1+x)^M-1}
\]
for $L\geq M$ and $x>0$, we see that
\[
\int_{\left(\log\frac1{1-r}\right)^{-1}\frac{4r^2}{(1-r^2)^2}}^{\frac{4r^2}{(1-r^2)^2}} \frac1{(1+x)^L-1} \frac{\sqrt{x}}{1+x}
\frac{dx}{\sqrt{1-\frac{(1-r^2)^2}{4r^2}x}}
\]
is uniformly $o(1)$ for all $L\geq 1$, say. Finally it is not hard to see that
\begin{equation*}
(1-r^2)^{2L-1}\left(\left(\log\frac1{1-r^2}\right)^{L-1/2}-1\right)=\left(1-(1-r^2)^{2L-1}\right)o(1)
\end{equation*}
as $L\rightarrow1/2^+$.

We conclude that
\[
I_L(r)= 4(1-r) \int_{0}^{\infty} \frac1{(1+x^2)^L-1} \frac{x^2}{1+x^2}\,dx (1+o(1))
\]
for fixed $L>1/2$, and the term $o(1)$ is uniformly small for all $L\geq 1$, while
\[
I_L(r)= \frac{2(1-r)}{L-\frac12}\left(1-(1-r)^{2L-1}\right)(1+o(1))
\]
as $L\rightarrow1/2^+$ and $r\rightarrow1^-$.

(c) We now assume that $L<1/2$. We now aim to show that the main contribution to \eqref{intinx} comes from the `big' values of $x$. Again making the
change of variables $y=\frac{(1-r^2)^2}{4r^2}x$ we see that, for fixed $L$,
\begin{align*}
\int_{\log\frac1{1-r}}^{\frac{4r^2}{(1-r^2)^2}} \frac1{(1+x)^L-1} &\frac{\sqrt{x}}{1+x} \frac{dx}{\sqrt{1-\frac{(1-r^2)^2}{4r^2}x}} \\
&= \int_{\log\frac1{1-r}}^{\frac{4r^2}{(1-r^2)^2}} \frac1{x^{L+1/2}} \frac{dx}{\sqrt{1-\frac{(1-r^2)^2}{4r^2}x}} (1 + o(1)) \\
&= (1-r)^{2L-1}\int_{\log\frac1{1-r}\frac{(1-r^2)^2}{4r^2}}^1 \frac 1{y^{L+1/2}} \frac{dy}{\sqrt{1-y}} (1 + o(1)).
\end{align*}
Moreover the term $o(1)$ can be taken to be uniform for $L$ close to $1/2$. We have
\begin{equation*}
\int_{\log\frac1{1-r}\frac{(1-r^2)^2}{4r^2}}^1 \frac 1{y^{L+1/2}} \frac{dy}{\sqrt{1-y}}=\int_{0}^1 \frac 1{y^{L+1/2}}
\frac{dy}{\sqrt{1-y}}-\int_0^{\log\frac1{1-r}\frac{(1-r^2)^2}{4r^2}} \frac 1{y^{L+1/2}} \frac{dy}{\sqrt{1-y}}
\end{equation*}
and these integrals converge for $L<1/2$. Now
\[
\int_{0}^1 \frac 1{y^{L+1/2}} \frac{dy}{\sqrt{1-y}} =B\Big(\frac12-L,\frac12\Big) =\frac{\Gamma(\frac12-L)\Gamma(\frac12)}{\Gamma(1-L)}
=\frac{\Gamma(\frac12-L)\sqrt{\pi}}{\Gamma(1-L)}
\]
where, again, $B$ is the Beta function. As $L\rightarrow1/2^-$ we have
\begin{equation*}
\frac{\Gamma(\frac12-L)\sqrt{\pi}}{\Gamma(1-L)}=\frac1{\frac12-L}+O(1).
\end{equation*}
Also
\begin{align*}
\int_0^{\log\frac1{1-r}\frac{(1-r^2)^2}{4r^2}} \frac 1{y^{L+1/2}} \frac{dy}{\sqrt{1-y}}&= \int_0^{\log\frac1{1-r}\frac{(1-r^2)^2}{4r^2}} \left(\frac
1{y^{L+1/2}} + O(y^{1/2-L})\right) \,dy\\&=\frac{(1-r)^{1-2L}}{\frac12-L}\left(\log\frac1{1-r}\right)^{\frac12-L}+o(1),
\end{align*}
which is $o(1)$ for fixed $L$. We therefore have
\begin{equation*}
\int_{\log\frac1{1-r}}^{\frac{4r^2}{(1-r^2)^2}} \frac1{(1+x)^L-1} \frac{\sqrt{x}}{1+x} \frac{dx}{\sqrt{1-\frac{(1-r^2)^2}{4r^2}x}}
=(1-r)^{2L-1}\frac{\Gamma(\frac12-L)\sqrt{\pi}}{\Gamma(1-L)}(1 + o(1))
\end{equation*}
for fixed $L$, while
\begin{equation*}
\int_{\log\frac1{1-r}}^{\frac{4r^2}{(1-r^2)^2}} \frac1{(1+x)^L-1} \frac{\sqrt{x}}{1+x} \frac{dx}{\sqrt{1-\frac{(1-r^2)^2}{4r^2}x}}
=\frac{(1-r)^{2L-1}}{\frac12-L} (1-(1-r)^{1-2L})(1 + o(1))
\end{equation*}
as $L\rightarrow1/2^-$ and $r\rightarrow1^-$.

It remains to show that the remaining parts of \eqref{intinx} are small in comparison. Now
\[
\int_0^1 \frac1{(1+x)^L-1} \frac{\sqrt{x}}{1+x} \frac{dx}{\sqrt{1-\frac{(1-r^2)^2}{4r^2}x}} = O(1)
\]
and
\begin{align*}
\int_1^{\log\frac1{1-r}} \frac1{(1+x)^L-1} \frac{x^{3/2}}{(1+x)^2} \frac{dx}{\sqrt{1-\frac{(1-r^2)^2}{4r^2}x}} &\lesssim \int_1^{\frac{1}{1-r^2}}
\frac1{x^{L+1/2}}\,dx\\
&= \frac1{\frac12-L}\left(\left(\log\frac1{1-r}\right)^{1/2-L}-1\right).
\end{align*}
We therefore have
\[
I_L(r)= \frac{2\sqrt{\pi}\Gamma(\frac12-L)}{\Gamma(1-L)}(1-r)^{2L} (1 + o(1))
\]
for fixed $L<1/2$ and
\[
I_L(r)= \frac{2(1-r)^{2L}}{\frac12-L}(1-(1-r)^{1-2L}) (1 + o(1))
\]
as $L\rightarrow1/2^-$ and $r\rightarrow1^-$.

We now suppose that $L\rightarrow0^+$ and furthermore that $\frac L{1-r}\to\infty$, or equivalently, that $\E[n_L(r)]\to\infty$. We write
\begin{equation*}
\delta=\delta(L,r)=\exp\left(-\left(\frac{1-(1-r)^{2L}}{L}\right)^{1/2}\right)\to0
\end{equation*}
and note that
\begin{equation*}
\frac{\delta}{(1-r)^2}=\exp\left(-\left(\frac{1-(1-r)^{2L}}{L}\right)^{1/2}+2\log\frac1{1-r}\right)\to\infty.
\end{equation*}
Now
\begin{align*}
\frac{\delta^L}{(1-r)^{2L}}-1 &= \frac{1-(1-r)^{2L}}{(1-r)^{2L}}\left(1+\frac{\delta^L-1}{1-(1-r)^{2L}}\right)\\
&= \frac{1-(1-r)^{2L}}{(1-r)^{2L}}\left(1+O\left(\frac{L}{1-(1-r)^{2L}}\right)^{1/2}\right)\\
&= \frac{1-(1-r)^{2L}}{(1-r)^{2L}}(1+o(1))
\end{align*}
which means that
\begin{equation*}
(1+x)^L-1=\frac{1-(1-r)^{2L}}{(1-r)^{2L}}(1+o(1))
\end{equation*}
for all $\frac{\delta}{(1-r)^2}<x<\frac{4r^2}{(1-r^2)^2}$, where the term $o(1)$ is uniform in $L$ and $r$. We therefore have, with the same change of
variables $y=\frac{(1-r^2)^2}{4r^2}x$,
\begin{align*}
\int_{\frac{\delta}{(1-r)^2}}^{\frac{4r^2}{(1-r^2)^2}} \frac1{(1+x)^L-1} \frac{\sqrt{x}}{1+x} &\frac{dx}{\sqrt{1-\frac{(1-r^2)^2}{4r^2}x}} \\
&= \frac{(1-r)^{2L}}{1-(1-r)^{2L}}\int_{\frac{\delta}{(1-r)^2}}^{\frac{4r^2}{(1-r^2)^2}} \frac1{\sqrt{x}} \frac{dx}{\sqrt{1-\frac{(1-r^2)^2}{4r^2}x}} (1 + o(1)) \\
&= \frac{(1-r)^{2L-1}}{1-(1-r)^{2L}}\int_{\frac{(1+r)^2}{4r^2}\delta}^1 \frac 1{\sqrt{y}} \frac{dy}{\sqrt{1-y}} (1 + o(1))\\
&= \frac{(1-r)^{2L-1}}{1-(1-r)^{2L}}B(1/2,1/2) (1 + o(1))\\
&= \pi\frac{(1-r)^{2L-1}}{1-(1-r)^{2L}} (1 + o(1)).
\end{align*}

We now show that the remaining contributions to \eqref{intinx} are of smaller order. We have
\[
\int_0^1 \frac1{(1+x)^L-1} \frac{\sqrt{x}}{1+x} \frac{dx}{\sqrt{1-\frac{(1-r^2)^2}{4r^2}x}} \simeq \int_0^1 \frac1{L\log(1+x)} \sqrt{x} \,dx =  O(L^{-1}).
\]
Using our hypothesis that $\frac L{1-r}\to\infty$ we see that
\begin{equation*}
  \frac1L=\frac{(1-r)^{2L-1}}{1-(1-r)^{2L}}o(1).
\end{equation*}
Also
\begin{align*}
\int_{1}^{\frac{\delta}{(1-r)^{2-4L}}} \frac1{(1+x)^L-1} \frac{\sqrt{x}}{1+x} \frac{dx}{\sqrt{1-\frac{(1-r^2)^2}{4r^2}x}} &\lesssim
\frac1L \int_1^{\frac{\delta}{(1-r)^{2-4L}}} \frac{dx}{\sqrt{x}} \\
&\lesssim \frac1L\sqrt{\delta}(1-r)^{2L-1}\\
&= \frac{(1-r)^{2L-1}}{1-(1-r)^{2L}}o(1).
\end{align*}
Finally, if $1/2<(1-r)^L<1$, then we have
\begin{align*}
\int_{\frac{\delta}{(1-r)^{2-4L}}}^{\frac{\delta}{(1-r)^2}} \frac1{(1+x)^L-1} \frac{\sqrt{x}}{1+x} \frac{dx}{\sqrt{1-\frac{(1-r^2)^2}{4r^2}x}} &\lesssim
\frac1L \int_{\frac{\delta}{16(1-r)^2}}^{\frac{\delta}{(1-r)^2}} \frac{dx}{\sqrt{x}} \\
&\lesssim \frac1L\sqrt{\delta}(1-r)^{2L-1}\\
&= \frac{(1-r)^{2L-1}}{1-(1-r)^{2L}}o(1),
\end{align*}
while if $0<(1-r)^L<1/2$ then, since
\begin{equation*}
\frac{\delta^L}{(1-r)^{2L-4L^2}}\geq\frac{\delta^L}{(1-r)^{L}}\geq\frac32
\end{equation*}
for L sufficiently small, we have
\begin{equation*}
(1+x)^L-1\gtrsim x^L
\end{equation*}
for all $x>\frac{\delta}{(1-r)^{2-4L}}$ and so
\begin{align*}
\int_{\frac{\delta}{(1-r)^{2-4L}}}^{\frac{\delta}{(1-r)^2}} \frac1{(1+x)^L-1} \frac{\sqrt{x}}{1+x} \frac{dx}{\sqrt{1-\frac{(1-r^2)^2}{4r^2}x}} &\lesssim
\int_{\frac{\delta}{(1-r)^{2-4L}}}^{\frac{\delta}{(1-r)^2}} \frac{dx}{x^{L+1/2}} \\
&\lesssim \delta^{1/2-L}(1-r)^{2L-1}\\
&= \frac{(1-r)^{2L-1}}{1-(1-r)^{2L}}o(1).
\end{align*}

We have therefore shown that
\[
I_L(r)= 2\pi\frac{(1-r)^{2L}}{1-(1-r)^{2L}} (1 + o(1))
\]
as $L\rightarrow0^+$, $r\rightarrow1^-$ and $\frac L{1-r}\to\infty$.

(b) We finally consider the critical case $L=1/2$; the integral we want to estimate is (see \eqref{intinx})
\[
\int_{0}^{\frac{4r^2}{(1-r^2)^2}} \frac1{\sqrt{1+x}-1} \frac{\sqrt{x}}{1+x} \frac{dx}{\sqrt{1-\frac{(1-r^2)^2}{4r^2}x}}.
\]
It is clear that
\[
\int_{0}^{1} \frac1{\sqrt{1+x}-1} \frac{\sqrt{x}}{1+x} \frac{dx}{\sqrt{1-\frac{(1-r^2)^2}{4r^2}x}} = O(1).
\]
and that
\begin{align*}
\int_{1}^{\log\frac1{1-r}} \frac1{\sqrt{1+x}-1} \frac{\sqrt{x}}{1+x} \frac{dx}{\sqrt{1-\frac{(1-r^2)^2}{4r^2}x}} &\lesssim
\int_{1}^{\log\frac1{1-r}} \frac1{x}\,dx\\
&= \log\log\frac1{1-r}\\
&= o(\log\frac1{1-r}).
\end{align*}
We finally compute that
\[
\int_{\log\frac1{1-r}}^{\frac{4r^2}{(1-r^2)^2}} \frac1{\sqrt{1+x}-1} \frac{\sqrt{x}}{1+x} \frac{dx}{\sqrt{1-\frac{(1-r^2)^2}{4r^2}x}} =
\int_{\log\frac1{1-r}}^{\frac{4r^2}{(1-r^2)^2}} \frac1{x} \frac{dx}{\sqrt{1-\frac{(1-r^2)^2}{4r^2}x}} (1+o(1))
\]
Once more making the change of variables $y=\frac{(1-r^2)^2}{4r^2}x$ we see that
\begin{align*}
\int_{\log\frac1{1-r}}^{\frac{4r^2}{(1-r^2)^2}} \frac1{x} \frac{dx}{\sqrt{1-\frac{(1-r^2)^2}{4r^2}x}} &=
\int_{\frac{(1-r^2)^2}{4r^2}\log\frac1{1-r}}^1 \frac 1y \frac{dy}{\sqrt{1-y}}\\
&= \left.\log\frac{1-\sqrt{1-y}}{1+\sqrt{1-y}}\right|_{\frac{(1-r^2)^2}{4r^2}\log\frac1{1-r}}^1\\
&= \log(\frac{4r^2}{(1-r^2)^2\log\frac1{1-r}}) + O(1)\\
&= 2\log\frac 1{1-r}(1+o(1)).
\end{align*}
\end{proof}

To prove Theorem~\ref{critvals}~(a) we note that, since
\begin{equation*}
  \lim_{x\to\infty}\frac{x^a\Gamma(x)}{\Gamma(x+a)}=1
\end{equation*}
for any real $a$, we have for $L$ large
\[
\frac{\Gamma(Ln-\frac12)}{\Gamma(Ln+1)}=\frac{\Gamma(Ln-\frac12)}{Ln\Gamma(Ln)}=(Ln)^{-3/2}(1+o(1))
\]
where the error term $o(1)$ is uniform in $n$, and so
\[
\sum_{n=1}^\infty\frac{\Gamma(Ln-\frac12)}{\Gamma(Ln+1)}=L^{-3/2}\sum_{n=1}^\infty n^{-3/2}(1+o(1)).
\]

\section*{Acknowledgements}
This work was undertaken as part of my PhD studies, which were co-supervised by Xavier Massaneda and Joaquim Ortega-Cerd\`a. Mikhail Sodin first suggested
I work on this problem. I am grateful to all three of them, and also to Ron Peled and Kristian Seip for many useful discussions and suggestions. I first
began this work when I was a visitor at the Department of Mathematical Sciences, NTNU, Trondheim. I thank them for their hospitality.

\end{document}